\documentclass[a4paper,11pt]{article}
\usepackage[utf8]{inputenc}

\usepackage{amsthm,amsmath,amsfonts,amssymb}
\usepackage[english]{babel}
\usepackage{graphicx}
\usepackage{verbatim}
\usepackage{color}
\usepackage{url}
\usepackage{soul}

\DeclareMathOperator{\diag}{diag}

\definecolor{orange}{RGB}{255,127,0}

\newcommand\blfootnote[1]{%
  \begingroup
  \renewcommand\thefootnote{}\footnote{#1}%
  \addtocounter{footnote}{-1}%
  \endgroup
}

\newcommand{\mytilde}{\raise.17ex\hbox{$\scriptstyle\mathtt{\sim}$}}
\def\Cay{\mathop{\rm Cay }\nolimits}
\def\Circ{\mathop{\rm Circ }\nolimits}

\def\det{\mathop{\rm det }\nolimits}

\def\diag{\mathop{\rm diag }\nolimits}

\def\mod{\mathop{\rm mod }\nolimits}

\def\Z{\ns{Z}}

\def\Z{\ns Z}

\def\e{\mbox{\boldmath $e$}}

\def\p{\mbox{\boldmath $p$}}

\def\vec0{\mbox{\boldmath $0$}}

\def\G{\Gamma}

\def\Z{\ns{Z}}

\def\G{\Gamma}

\def\Re{\mathbb R}
\def\Z{\mathbb Z}

\theoremstyle{plain}   % Cal carregar el paquet theorem.sty o amsthm.sty

\newtheorem{theorem}{Theorem}[section]
\newtheorem{proposition}[theorem]{Proposition}

\newtheorem{lemma}[theorem]{Lemma}

\begin{document}

\title{An improved Moore bound and some new optimal\\ families of mixed Abelian Cayley graphs}

\vskip-.25cm

\author{C. Dalf\'o$^a$, M. A. Fiol$^b$, N. L\'opez$^c$, J. Ryan$^d$\\
{\small $^a$Dept. de Matem\`atica, Universitat de Lleida}\\
{\small Igualada (Barcelona), Catalonia}\\
{\small {\tt cristina.dalfo@udl.cat}}\\
{\small $^{b}$Dept. de Matem\`atiques, Universitat Polit\`ecnica de Catalunya} \\
{\small Barcelona Graduate School of Mathematics} \\
{\small Barcelona, Catalonia} \\
{\small {\tt miguel.angel.fiol@upc.edu}} \\
{\small $^c$Dept. de Matem\`atica, Universitat de Lleida}\\
%C/ Jaume II 69, 25001
{\small Lleida, Spain}\\
{\small {\tt nacho.lopez@udl.cat}}\\
$^d${\small School of Electrical Engineering and Computer Science}\\
{\small The University of Newcastle,}\\
{\small Newcastle, Australia}\\
{\small{\tt {joe.ryan@newcastle.edu.au}}}
}

\date{}
\maketitle

\vskip-.5cm
\begin{abstract}
We consider the case in which mixed
graphs (with both directed and undirected edges) are Cayley graphs of Abelian groups. In this case, some Moore bounds were derived for the maximum number of vertices that such graphs can attain. We first show these bounds can be improved  if we know more details about the order of some elements of the generating set. Based on these improvements, we  present some new families of mixed graphs. For every fixed value of the degree, these families have an asymptotically large number of vertices as the diameter increases. In some cases, the results obtained are shown to be optimal.
\end{abstract}

\vskip-.25cm


\vskip-.25cm

\blfootnote{
	\begin{minipage}[l]{0.3\textwidth} \includegraphics[trim=10cm 6cm 10cm 5cm,clip,scale=0.15]{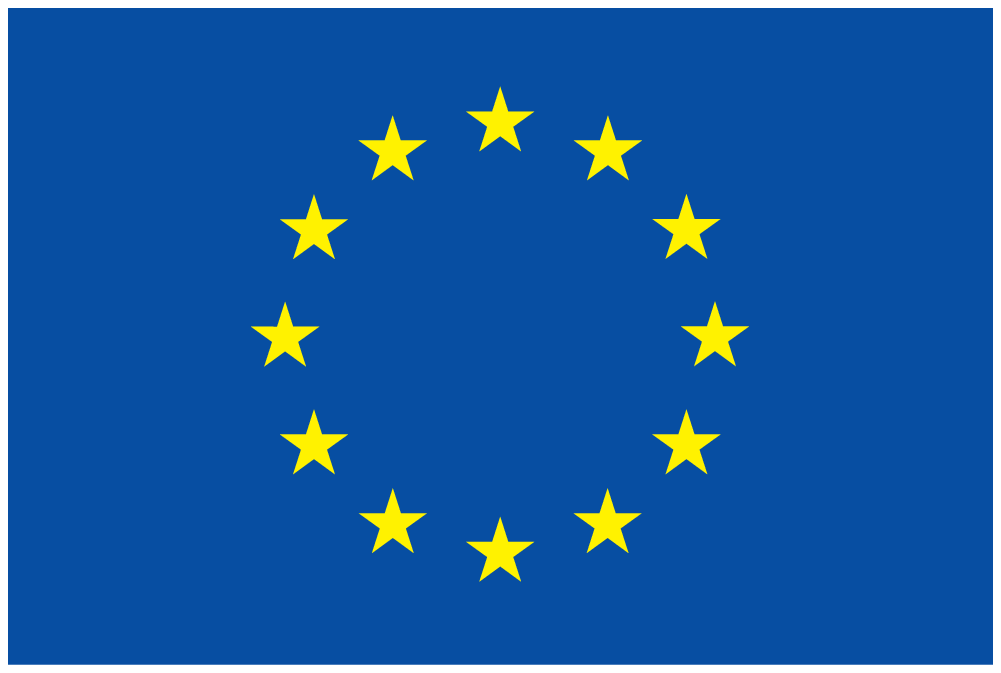} \end{minipage}  \hspace{-2cm} \begin{minipage}[l][1cm]{0.79\textwidth}
		The research of C. Dalf\'o has also received funding from the European Union's Horizon 2020 research and innovation programme under the Marie Sk\l{}odowska-Curie grant agreement No 734922.
\end{minipage}}

\section{Preliminaries}\label{sec:intro}

The degree/diameter or $(d,k)$ problem seeks to determine the largest possible graph (in terms of the number of vertices), for a given maximum degree and a given diameter. This problem has been considered for different families of graphs. For instance: bipartite graphs in Dalf\'o, Fiol, and L\'opez \cite{DFL18}; planar graphs in Fellows, Hell,  and Seyfarth \cite{Fell95} and in Tischenko \cite{Ti12}; vertex-transitive graphs in Machbeth, \v{S}iagiov\'a, \v{S}ir\'a\v{n}, and Vetr\'{\i}k \cite{Mac10}, and in \v{S}iagiov\'a and Vetr\'{\i}k\cite{Sia07}; Cayley graphs also in \cite{Mac10,Sia07}, and in  Vetr\'{\i}k \cite{Ve13}; Cayley graphs of Abelian groups in Dougherty and Faber \cite{Dou04}, and Aguil\'o, Fiol and P\'erez \cite{afp16}; and circulant graphs in Wong and Coppersmith \cite{Wong74}, 
Morillo, Fiol, and  F\`abrega \cite{mff85},  Fiol,  Yebra, Alegre, and Valero \cite{fyav87}, and in Monakhova \cite{Mona12}. For more information, see the comprehensive survey of Miller and \v{S}ir\'a\v{n} \cite{Moore-survey}. Here we deal with the case of the so-called mixed graphs. A mixed graph $G=(V,E,A)$, on a vertex set $V$, has edge set $E$ and arc set $A$. That is, we consider the presence of both undirected edges and directed edges.
Then, in the degree/diameter problem for mixed graphs, we have three parameters: a maximum undirected degree $r$, a maximum directed out-degree $z$, and a diameter $k$.

A natural upper bound for the maximum number of vertices $M(r,z,k)$ for a mixed graph under such degrees and diameter restrictions is (see Buset, El Amiri, Erskine, Miller, and P\'erez-Ros\'es \cite{mixedmoore}):
\begin{equation}\label{eq:new}
M(z,r,k) = A \frac{u_1^{k+1}-1}{u_1-1}+B\frac{u_2^{k+1}-1}{u_2-1}
\end{equation}
where, with $d=r+z$ and $v=(d-1)^2+4z$,
\begin{align}
u_1 &=\displaystyle{\frac{d-1-\sqrt{v}}{2}}, \qquad
u_2 =\displaystyle{\frac{d-1+\sqrt{v}}{2}}, \label{u's}\\
A   &=\displaystyle{\frac{\sqrt{v}-(d+1)}{2\sqrt{v}}}, \quad \
B   =\displaystyle{\frac{\sqrt{v}+(d+1)}{2\sqrt{v}}}. \label{A&B}
\end{align}

%Besides this general bound given above, researchers are also interested in some particular versions of the problem, namely when the graphs are restricted to a certain class, such as the class of bipartite graphs \cite{DFL18}, planar graphs \cite{Fell95,Ti12}, vertex-transitive graphs \cite{Mac10,Sia07}, Cayley graphs \cite{Mac10,Sia07,Ve13}, Cayley graphs of Abelian groups \cite{Dou04}, or circulant graphs \cite{Wong74,Mona12}. In this paper we are concerned with mixed Abelian Cayley graphs.
%
%For most of these graph classes there exist Moore-like upper bounds, which in general are smaller than the Moore bound for general graphs, although some of them are quite close to the Moore bound. For example, the Moore-like upper bound for bipartite mixed graphs is (when $r>0$):
%\begin{equation}
%\label{Moore-mix-bip}
%M_B(r,z,k)=
%2\left(A\,\frac{u_1^{k+1}-u_1}{u_1^2-1}+ B\,\frac{u_2^{k+1}-u_2}{u_2^2-1}\right),
%\end{equation}
%where $u_1$, $u_2$, $A$, and $B$ are given by \eqref{u's} and \eqref{A&B} (see Dalf\'o, Fiol, and L\'opez \cite{DFL18}).

The upper bound on the maximum number of vertices for mixed Abelian Cayley graphs was given by
L\'opez, P\'erez-Ros\'es, and Pujol\`as in \cite{LOPEZ2016145,lpp17} by using recurrences and generating functions. Let $\Gamma$ be an Abelian group, and let $\Sigma$ be a generating set of $\Gamma$ containing $r_{\alpha}$ involutions and $r_{\omega}$ pairs of generators and their inverses, and $z_{\omega}$ additional generators, whose inverses are not in $\Sigma$ (the  $r_{\omega}$ pairs and $z_{\omega}$ generators have undetermined orders). Thus, the Cayley graph $\Cay(\Gamma, \Sigma)$ is a mixed graph with undirected degree $r=r_{\alpha}+2r_{\omega}$, and directed out-degree $z=z_{\omega}$.
An upper bound for the number of vertices of $\Cay(\Gamma, \Sigma)$, as a function of the diameter $k$, is
\begin{equation}
\label{eq:upper1}
M_{AC}(r_{\alpha},r_{\omega},z_{\omega},k)=\sum_{i=0}^k {r_{\omega}+z_{\omega}+i \choose i}{r_{\alpha}+r_{\omega} \choose k-i}.
\end{equation}
Recently, Dalf\'o, Fiol, and L\'opez \cite{dfl18}, by using a more direct combinatorial reasoning, obtained the following alternative expression for the same bound.
\begin{equation}
\label{eq:upper2}
M_{AC}(r_{\alpha},r_{\omega},z_{\omega},k)=\sum_{i=0}^{r_{\omega}}{r_{\omega}\choose i}2^i\sum_{j=0}^{r_{\alpha}}{r_{\alpha}\choose j}{k+z_{\omega}-j\choose i+z_{\omega}}.
\end{equation}

In particular, \eqref{eq:upper2} yields the known Moore bounds for the Abelian Cayley digraphs $(r_{\alpha}=r_{\omega}=0)$, and Abelian Cayley graphs with no involutions $(r_{\alpha}=z_{\omega}=0)$. Namely,
$$
M_{AC}(0,0,z_{\omega},k)={k+z_{\omega}\choose z_{\omega}}\qquad\mbox{and}\qquad
M_{AC}(0,r_{\omega},0,k)=\sum_{i=0}^{r_{\omega}} 2^i{r_{\omega}\choose i}{k\choose i},
$$
respectively. See Wong and Coppersmith \cite{Wong74} for the former, and Stanton and Cowan \cite{sc70} for the latter.

Dalf\'o, Fiol, and L\'opez \cite{dfl19} proved that the Moore bound on mixed Abelian Cayley graphs satisfies some symmetries.

\begin{lemma}[\cite{dfl19}]
	For any integer $\nu$ such that $-r_2\le \nu\le \min\{r_1,z\}$,  the Moore bound for the mixed Abelian Cayley graphs satisfies
	\begin{equation}
	\label{symmetry}
	M_{AC}(r_1,r_2,z,k)=M_{AC}(r_1-\nu,r_2+\nu,z-\nu,k).
	\end{equation}
\end{lemma}

\subsection{Abelian Cayley graphs from congruences in $\mathbb{Z}^n$}
\label{smith}

Let $M$ be a $n\times n$ nonsingular integral matrix, and $\Z^{n}$  the additive group of $n$-vectors with
integral components. The set $\Z^{n}M$, whose elements are linear combinations (with integral coefficients) of the rows of $M$ is said to be the lattice generated by $M$.
The concept of congruence in $\Z$ has the following natural
generalization to $\Z^{n}$ (see Fiol \cite{f87}). Let $u,v\in \Z^{n}$. We say that
{\it $u$  is congruent with $v$ modulo $M$}, denoted by
$u \equiv v\pmod{M}$, if
\begin{equation}
\label{eq2}
u-v \in \Z^{n}M.
\end{equation}
The Abelian quotient group $\Z^{n}/\Z^{n}M$ is referred to as the
{\it group of integral vectors modulo $M$}.
In particular,
when $M=\diag (m_{1}, \ldots ,m_{n})$, the group $\Z^{n}/\Z^{n}M$ is the direct product of the cyclic
groups $\Z_{m_{i}}$, for $i=1, \ldots ,n$.

Let $M$ be an $n \times n$ integral matrix as above.
Let $A=\{a_{1},\ldots, a_{d}\} \subseteq \Z^{n}/\Z^{n}M$. The {\it multidimensional
	$($d-step$)$ circulant digraph $G(M,A)$} has as vertex set the integral
vectors modulo $M$, and every vertex $u$ is adjacent to the vertices
$u+A\pmod{M}$. As in the case of digraphs, the {\it
	multidimensional $($d-step$)$ circulant graph $G(M,A)$} is defined similarly just requiring $A=-A$.
Clearly, a multidimensional circulant (digraph, graph, or mixed graph) is a Cayley graph of the Abelian group $\Gamma=\Z^{n}/\Z^{n}M$.
In our context, if $\G$ is an Abelian group with generating set $\Sigma$ containing $r_{\alpha}+2r_2+z$ generators (with the same notation as before), then there exists an integer $n\times n$ matrix $M$ with size $n=r_{\alpha}+r_2+z$ such that
$$
\Cay(\G,\Sigma)\cong \Cay(\Z^{n}/\Z^{n}M,\Sigma'\},
$$
where $\Sigma'=\{e_1,\ldots,e_{r_{\alpha}},\pm e_{r_{\alpha}+1},\ldots,\pm e_{r_{\alpha}+r_2},e_{r_{\alpha}+r_2+1},\ldots,e_{r_{\alpha}+r_2+z}\}$, and the $e_i$'s stand for the unitary coordinate vectors. For example,  the two following Cayley mixed graphs
$$
\Cay(\Z_{24},\{\pm 2,3,12\}),\quad \mbox{and}\quad
\Cay(\Z^3/\Z^3 M, \{\pm e_1,e_2,e_3\}) \mbox{\ with\ }
M=\left(
\begin{array}{ccc}
3 & -2 & 0\\
0 & 4 & 1\\
0 & 0 & 2
\end{array}
\right),
$$
are isomorphic since the Smith normal form of $M$ is $S=\diag(1,1,24)$ and
$$
S=UMV=\left(
\begin{array}{ccc}
-1 & 0 & 0\\
-4 & 1 & 0\\
-8 & 2 & -1
\end{array}
\right)\left(
\begin{array}{ccc}
3 & -2 & 0\\
0 & 4 & 1\\
0 & 0 & 2
\end{array}
\right)\left(
\begin{array}{ccc}
-1 & 0 & 2\\
-1 & 0 & 3\\
0 & 1 & -12
\end{array}
\right).
$$
Indeed, $\Z^3/\Z^3 M$ is a cyclic group of order $|\det M|=24$ and
%, according to \eqref{eq5},
the generators $\pm e_1$, $e_2$, and $e_3$ of $\Z^3/\Z^3 M$ give rise to the generators $\pm 2$, $3$, and $-12=12$ $(\mod 24)$ of $\Z_{24}$; see the last column of $V$. For more details, see the paper \cite{dfl19} by the authors.

The following basic result is a simple consequence of the close relationship between the Cartesian product of Abelian Cayley graphs
and the direct products of Abelian groups (see, for instance, Fiol \cite{f87,f95}).

\begin{lemma}[\cite{f87,f95}]
	\label{basic-lemma}
	\begin{itemize}
		\item[$(i)$]
		The Cartesian product  of the Abelian Cayley graphs $G_1=\Cay(\G_1,\Sigma_1)$ and $G_2=\Cay(\G_2,\Sigma_2)$
		is the Abelian Cayley graph $G_1\times G_2=\Cay(\G_1\times \G_2,(\Sigma_1,0)\cup (0,\Sigma_2))$.
		In terms of congruences, if $\G_1=\Z^{n_1}/\Z^{n_1}M_1$,  $\G_2=\Z^{n_2}/\Z^{n_2}M_2$, $\Sigma_1=\{e_1,\ldots,e_{n_1}\}$, and
		$\Sigma_2=\{e_1,\ldots,e_{n_2}\}$, then
		$G_1\times G_2=\Cay(\Z^{n_1+n_2}/\Z^{n_1+n_2}M, \Sigma)$, where $M$ is the block-diagonal matrix $\diag(M_1,M_2)$ and $\Sigma=\{e_1,\ldots,e_{n_1+n_2}\}$.
		\item[$(ii)$]
		Let us consider the Cayley Abelian graph $G=\Cay(\G,\{a_1,\ldots,a_n,b\})$ with diameter $D$, where $b$ is an involution. Then, the quotient graph $G'=G/K_2$, obtained from $G$ by contracting all the edges generated by $b$, is an Abelian Cayley graph on the quotient group $\G/\Z_2$, with $n$ generators and diameter $D'\in\{D-1,D\}$.
		\item[$(iii)$]
		For a given integer matrix $M$ with a row $u$, let $G=\Cay(\Z^n/\Z^nM,\{e_1,\ldots,e_n\})$ have diameter $D$. Then, for a integer $\alpha>1$, the graph $G'=\Cay(\Z^n/\Z^nM',\{e_1,\ldots,e_n,$ $2u,\ldots,\alpha u\})$, where $M'$ is obtained from $M$ multiplying $u$ by $\alpha$, has diameter $D'=D+1$.
	\end{itemize}
\end{lemma}

\section{An improved bound}\label{sec:bound}

The above bound $M_{AC}(r_{\alpha},r_{\omega},z_{\omega},k)$ can be improved if we know more details about the order of some elements of the generating set.

\begin{theorem}
Let $\G$ be an Abelian group with generating set $\Sigma$ containing:
\begin{itemize}
\item
$r_{\alpha}$ involutions;
\item
$r_s$ elements of order $2s+1$, for $s=1,2,\ldots,m(\le k)$, together with their inverses;
\item
$r_{\omega}$ elements of undetermined order greater than $2k+1$, together with their inverses;
\item
$z_t$ elements of order $t+1$, for $t=2,\ldots,n(\le k)$, without their inverses (note that $z_1$ would be the same as $r_{\alpha}$);
\item
$z_{\omega}$ elements of undetermined order greater than $k+1$, without their inverses.
\end{itemize}
Then, the Cayley graph $\Cay(\Gamma, \Sigma)$ is a mixed graph with undirected degree $r=r_{\alpha}+2r_1+\cdots+2r_{m}+2r_{\omega}$, directed out-degree $z=z_2+z_3+\cdots+z_{n}+z_{\omega}$, and number of vertices $N$ satisfying the bound
\begin{align}
\label{eq:upper4}
N \le & M_{AC}(r_{\alpha},r_1,\ldots,r_{m},r_{\omega},z_2,\ldots,z_{n},z_{\omega},k) \nonumber\\
  = & \sum_{i_{\alpha}=0}^{r_{\alpha}}
     \sum_{i_{\omega}=0}^{r_{\omega}}
      {r_{\alpha} \choose i_{\alpha}}
      {r_{\omega} \choose i_{\omega}}2^{i_{\omega}}
      \sum_{i_1=0}^{r_1}\cdots \sum_{i_{m}=0}^{r_m}
 \prod_{h=1}^m\sum_{\sigma_1+\cdots+\sigma_h=i_h}{r_h\choose \sigma_1,\ldots,\sigma_h}2^{i_1+\cdots +i_m}\nonumber\\
 & %\times
 \sum_{j_2=0}^{z_2}\sum_{j_3=0}^{z_3}\cdots \sum_{j_n=0}^{z_n} \prod_{h=2}^n\sum_{\tau_1+\cdots+\tau_h=j_h}{z_h\choose \tau_1,\ldots,\tau_h} \\
  & %\times
  {z_{\omega}+k-i_{\alpha}-\sum_{s=1}^m\sum_{t=1}^s t\sigma_t-\sum_{s=1}^n\sum_{t=1}^s t\tau_t \choose i_{\omega}+z_{\omega}}.\nonumber
\end{align}
\end{theorem}

\begin{proof}
A vertex $u$ at distance at most $k$ from $0$ can be represented by the situation of $k$ balls (representing the presence/absence of the edges/arcs in the shortest path from $0$ to $u$) placed in $1+r_{\alpha}+\sum_{s=1}^m r_s+r_{\omega}+\sum_{t=0}^n z_t+z_{\omega}$ boxes (representing the presence/absence of the generators) with the following conditions:
\begin{itemize}
\item[$(i)$]
One box contains the number of (white) balls of the non-existing edges/arcs. Then, such a number is just the complement to $k$ of the sum of all the balls in the other boxes.
\item[$(ii)$]
Each of the $r_{\alpha}$ boxes contains at most {\bf one} (white) ball corresponding to the edge defined by the involution.
This makes the following total number of possibilities:
\begin{equation}
\label{ii}
\sum_{i_{\alpha}=0}^{r_{\alpha}} {r_{\alpha} \choose i_{\alpha}}.
\end{equation}
\item[$(iii)$]
Each of the $r_s$ boxes, with $s=1,\ldots,m$, contains a number of {\bf at most} $s$  balls, which are either all white or all black, of the edges defined by the corresponding generator $a$ (white) or $-a$ (black) with order $2s+1$. Then, if there are exactly $i_s\in\{0,1,\ldots,r_s\}$ boxes with at least one ball, and there are $\sigma_i$  of the boxes with exactly $i$  (either white or black)  balls, for $i=1,\ldots,s$, then we have $\sigma_1+2\sigma_2+\cdots+s\sigma_s$ balls in $i_s$ boxes with $\sigma_1+\cdots+\sigma_s=i_s$, and the total number of the possible situations for the $r_s$ boxes is
 \begin{equation}
\label{iii}
\sum_{i_s=0}^{r_s}2^{i_s}\sum_{\sigma_1+\cdots+\sigma_s=i_s}{r_s\choose \sigma_1,\sigma_2,\ldots,\sigma_s},
\end{equation}
where the last terms stand for the multinomial number $\frac{r_s!}{\sigma_1!\sigma_2!\cdots\sigma_s!}$. (The term $2^{i_{s}}$ accounts for the two possible colors of all balls in each of the $i_{s}$ boxes.)
\item[$(iv)$]
Each of the $r_{\omega}$ boxes contains a number of at most $k$  balls, which are either all white or all black, corresponding to the edges defined by the generators $a$ (white) or $-a$ (black). Then,  if there are exactly $i_{\omega}\in\{0,1,\ldots,r_{\omega}\}$  nonempty boxes, then we have, for the moment, having one ball in each box, a total of
 \begin{equation}
\label{iv}
\sum_{i_{\omega}=0}^{r_{\omega}}2^{i_\omega}{r_{\omega}\choose i_{\omega} }
\end{equation}
possible situations.
\item[$(v)$]
Each of the $z_t$ boxes, with $t=2,\ldots,n$, contains a number of at most $t$ (white) balls. Then, reasoning as in $(iii)$, now we have that the number of possible situations for the $z_t$ boxes is
 \begin{equation}
\label{v}
\sum_{j_t=0}^{z_t}\sum_{\tau_1+\cdots+\tau_t=j_t}{z_t\choose \tau_1,\tau_2,\ldots,\tau_t}.
\end{equation}
\item[$(vi)$]
Each of the $z_{\omega}$ boxes contains a number of at most $k$ (white) balls corresponding to the arcs defined by the  generator $b$ (with $-b\not\in \Sigma$).
% Then,  if there are exactly $j_{\omega}\in\{0,1,\ldots,z_{\omega}\}$  nonempty boxes we have, for the moment, a total of
% \begin{equation}
%\label{vi}
%\sum_{j_{\omega}=0}^{z_{\omega}}{z_{\omega}\choose j_{\omega} }
%\end{equation}
%possible situations.
%\label{vii}
\item[$(vii)$]
Finally, there are $k-i_{\alpha}-i_{\omega}-\sum_{s=1}^m\sum_{i=1}^s i\sigma_i-\sum_{t=1}^n\sum_{j=1}^t j\tau_j$ balls left, to be placed in $1+i_{\omega}+z_{\omega}$ boxes, which gives a total of
\begin{equation}
\label{vii}
 {z_{\omega}+k-i_{\alpha}-\sum_{s=1}^m\sum_{t=1}^s t\sigma_t-\sum_{s=1}^n\sum_{t=1}^s t\tau_t \choose i_{\omega}+z_{\omega}}.
\end{equation}
\end{itemize}

 Putting all together, we obtain \eqref{eq:upper4}.
%Suppose that exactly $i_{\omega}$ of the $r_{\omega}$ boxes and $i_{\alpha}$ of the $r_{\alpha}$ boxes are non-empty. This gives a total of ${r_{\omega}\choose i_{\omega}}2^{i_{\omega}}{r_{\alpha}\choose i_{\alpha}}$ possibilities. Analogously, assume that exactly $i_2$ of the $r_2$  boxes are non-empty, where $\sigma_1$ of them have 1 ball and $\sigma_2$ of them have 2 balls (where $\sigma_1+\sigma_2=i_2$). In this case, the number of possible situations is given by ${r_2\choose \sigma_1,\sigma_2}2^{i_2}$.
%Finally, there are $k-i_{\omega}-i_{\alpha}-(\sigma_1+2\sigma_2)$ balls left to be placed in $1+i_{\omega}+z_{\omega}$ boxes, which gives a total of ${k+z_{\omega}-i_{\alpha}-\sigma_1-2\sigma_2\choose i_{\omega}+z_{\omega}}$ situations. Putting all together, we obtain \eqref{eq:upper3}.
%
%The proof of the general case \eqref{eq:upper3} goes along the same lines of reasoning. Now we must use   the fact that the number of ways of placing at most $n$ (undistinguished) balls in $m$ boxes, with the condition that no box can contain more than $2$ balls, which equals the sum of the multinomial numbers
%$$
%\sum_{\sigma_1+\cdots+\sigma_s=n}{m\choose \sigma_1,\ldots,\sigma_s}.
%$$
\end{proof}

For instance, as an example, let us consider the particular case where $\Sigma$ contains $r_{\alpha}$, $r_{2}$, $r_{\omega}$, and $z_{\omega}$ elements as above. That is, apart from the elements with undetermined order, we have  $r_{\alpha}$ involutions and $r_2$ elements of order $5$ together with their inverses.
Then, \eqref{eq:upper4} becomes
\begin{align}
N & \le M_{AC}(r_{\alpha},r_2,r_{\omega},z_{\omega},k)\nonumber\\
 & =\sum_{i_{\alpha}=0}^{r_{\alpha}} \sum_{i_{\omega}=0}^{r_{\omega}} {r_{\alpha}\choose i_{\alpha}}
 {r_{\omega}\choose i_{\omega}}2^{i_{\omega}}\sum_{i_2=0}^{r_2}\sum_{\sigma_1+\sigma_2=i_2}{r_2\choose \sigma_1, \sigma_2} 2^{i_2}{z_{\omega}+k-i_{\alpha}-\sigma_1-2\sigma_2\choose i_{\omega}+z_{\omega}},\label{eq:upper3}
\end{align}
where ${r_2\choose \sigma_1, \sigma_2}={r_2\choose \sigma_1}{r_2-\sigma_1\choose \sigma_2}$ stands for the trinomial number. \\

%% Nacho
The new upper bound \eqref{eq:upper4} is also useful to reduce the number of groups $\Gamma$ and/or generators $\Sigma$ that we need to check to find optimal graphs for given parameters. For instance, let us focus on the case $r=1$ and any $z \geq 1$. Of course, the undirected part of the Abelian Cayley graph must be generated by an involution, that is,  $r_{\alpha}=1$ and $r_2=r_3=\dots=r_{\omega}=0$, and for the directed part $z=z_2+z_3+\cdots+z_{n}+z_{\omega}$ must be satisfied. Nevertheless, the maximum value for \eqref{eq:upper4} is obtained when $z_2=z_3=\dots=z_{n}=0$ and $z_{\omega}=z$, since in this case we obtain the general upper bound given by \eqref{eq:upper1}:
\[
M_{AC}(1,0,\dots,0,z,k)=\sum_{i_{\alpha}=0}^{1} {1 \choose i_{\alpha}} {z+k-i_{\alpha} \choose z} = \frac{2k+z}{k+z} {k+z \choose z}.
\]
In the particular case $z=2$, we obtain $M_{AC}=(k+1)^2$ and if in addition $k=7$, then we have $M_{AC}=64=2^6$. There are $11$ Abelian groups of order $64$, but four or them, namely $\mathbb{Z}_2^6$, $\mathbb{Z}_2^4 \times \mathbb{Z}_4$, $\mathbb{Z}_2^2 \times \mathbb{Z}_4 \times \mathbb{Z}_4$ and $\mathbb{Z}_4^3$, can be rejected to find a largest graph since they contain elements of order at most $4$, and the upper bound $M_{AC}$ decreases to $34$ (taking $r_{\alpha}=1$, $z_3=2$ and $k=7$ in \eqref{eq:upper4}). There is also a strong condition for the generating set $\Sigma$ in the remaining groups of order $64$, that is, $\Sigma$ should contain three elements of order at least $8$. All these restrictions can be used to reduce the computation time to find optimal graphs.

\section{Some new dense families}
\label{sec:requalto1}

In this section, we deal with two families of optimal graphs with $r_\alpha=1$.

\subsection{The case $(r_{\alpha},r_{\omega},z_{\omega})=(1,2,0)$}

\begin{figure}[t]
	\begin{center}
		\includegraphics[width=5cm]{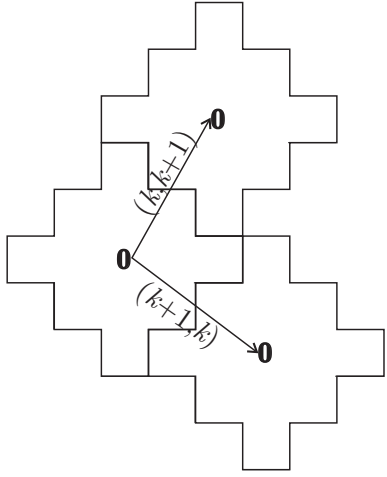}
	\end{center}
	\vskip-.25cm
	\caption{A discrete $\diamondsuit$-shaped tile and its tessellation.}
	\label{fig:diamond}
\end{figure}

To study this case, we can use the known  results of
circulant graphs with degree 4  given in Yebra, Fiol, Morillo, and Alegre \cite{yfma85}.
In that paper, the authors considered the case $r_{\alpha}=0$, $r_{\omega}=2$, and $z_{\omega}=0$, and it was proved that   the  Moore bound is $M_{AC}(0,2,0,k)=2k^2+2k+1$, in concordance with \eqref{eq:upper1} and \eqref{eq:upper2}. Besides, it was shown that the Moore bound is attained with the graphs $\Circ(\Z_N; \pm k, \pm(k+1))$, where $N=2k^2+2k+1$. Each of such graphs admits a representation like a discrete $\diamondsuit$-shaped tile with `radius' $k$ (see Figure \ref{fig:diamond}) that tessellates the plane.
Note that this shape is formed by unit squares centered at the integral points $\p\in \Re^2$ such that $\|\p\|_1\le k$. The corresponding lattice is $\Z^2M$ with matrix
$M=\left(\begin{array}{cc}
2k+1 & 1\\
k+1 & -k
\end{array} \right)$.

For the values considered here, the Moore bound is $M_{AC}(1,2,0,k)=4k^2+2$. However, this bound is not attainable, as shown in the following result.

\begin{figure}[t]
	\begin{center}
		\includegraphics[width=14cm]{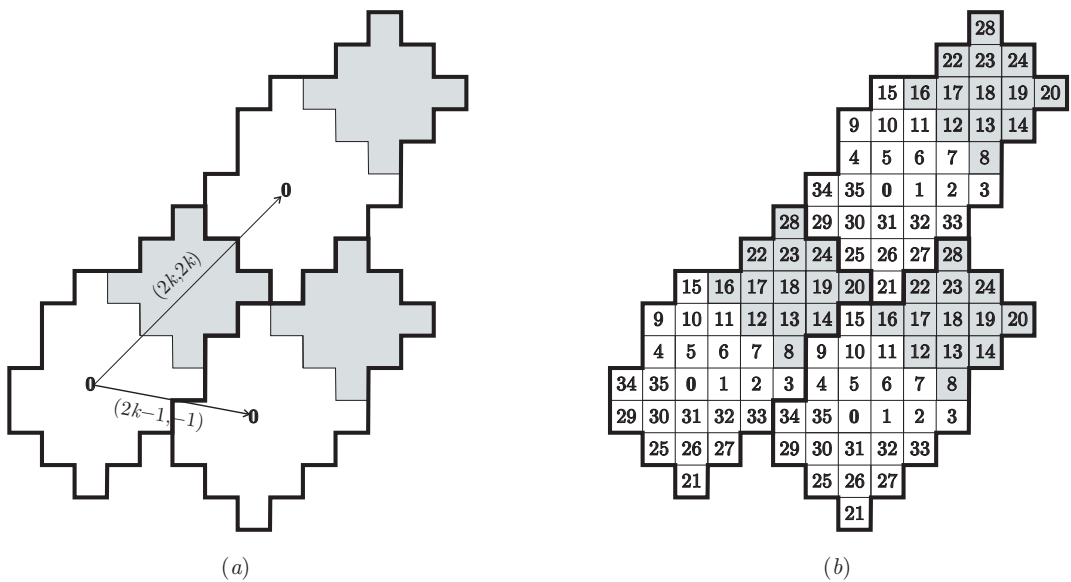}
	\end{center}
	\vskip-.5cm
	\caption{$(a)$ The general case for $N=4k^2$ in the proof of Proposition \ref{propo:diamonds};
			$(b)$ The particular case for $k=3$.}
	\label{fig:general-and-small-case-cross}
\end{figure}

\begin{proposition}
	\label{propo:diamonds}
	For any given diameter $k\ge 2$, the maximum order for a mixed Abelian Cayley graph with $r_{\alpha}=1$, $r_{\omega}=2$, and $z_{\omega}=0$ is $N=4k^2$, and the graph attaining it is $\Cay(N; \pm 1, \pm (2k-1), 2k^2)$.
\end{proposition}

\begin{proof}
	Let $V(0)$ and $V(\ast)$ be the sets of vertices at minimum distance from $0$, whose respective shortest paths do, or do not, contain vertex $\ast$ (the involution). Then, all the vertices in $V(\ast)$ must be at distance at most $k-1$ from the vertex $\ast$. Hence, $|V(\ast)|\le 2(k-1)^2+2(k-1)+1=2k^2-2k+1$, and the maximum is attained with a  $\diamondsuit$-shaped tile $\diamondsuit_{k-1}(\ast)$ of radius $k-1$. Similarly, $|V(0)|\le 2k^2+2k+1$, and the maximum corresponds to a $\diamondsuit$-shaped tile $\diamondsuit_{k}(0)$ of radius $k$. Thus, $M_{AC}(1,2,0,k)=|V(0)|+|V(\ast)|=M_{AC}(0,2,0,k)+M_{AC}(0,2,0,k-1)=4k^2+2$, as claimed. However, it is easy to check that any tile formed by $\diamondsuit_{k}(0)$ and $\diamondsuit_{k-1}(\ast)$ tessellates the plane. The optimal graphs are then obtained by considering the tile formed  by $\diamondsuit_{k}(0)$ minus two `extremal' vertices and $\diamondsuit_{k-1}(\ast)$, as shown in Figure \ref{fig:general-and-small-case-cross}$(a)$ (recall that the total number of vertices must be even).  Then, the lattice turns out to be 
	$\Z^2M$ with matrix
	$M=\left(\begin{array}{cc}
	2k & 2k\\
	2k-1 & -1
	\end{array} \right)$,
	having Smith normal form 
	$$
	S=\diag(1,4k^2)=UMV=\left(\begin{array}{cc}
	0 & -1\\
	1 & 2k
	\end{array} \right)\left(\begin{array}{cc}
	2k & 2k\\
	2k-1 & -1
	\end{array} \right)\left(\begin{array}{cc}
	0 & 1\\
	1 & 2k-1
	\end{array} \right).
	$$
	Then, according to the results of Subsection \ref{smith}, the $r_{\beta}=2$ steps are $\pm 1$ and $\pm (2k-1)$; see last column of $V$. This, together with the involution $2k^2$, completes the proof.
	For example, with $k=3$, we get $N=36$ for the graph $\Circ(36;\{\pm1,\pm 5, 18\})$; see Figure \ref{fig:general-and-small-case-cross}$(b)$.
\end{proof}

\subsection{The case $(r_{\alpha},r_{\omega},z_{\omega})=(1,1,1)$}

Reasoning as in Dalf\'o, Fiol, and L\'opez \cite{dfl18}, we can obtain the optimal constructions for mixed Abelian Cayley graphs with  $r_{\alpha}=r_{\omega}=z_{\omega}=1$, that is, graphs with undirected degree $3$ and directed degree $1$.
In this case, our study is based on
the results of Morillo and Fiol \cite{mf86} dealing with the case $r_{\alpha}=0$, $r_{\omega}=1$, and $z_{\omega}=1$. In this context, they proved that the (unattainable) Moore bound is $M_{AC}(0,1,1,k)=(k+1)^2$, in concordance with \eqref{eq:upper1} and \eqref{eq:upper2}. Besides, by using plane tessellations with $T$-shaped tiles (see Figure \ref{fig:general-cases}), it was shown that the maximum number of vertices for such graphs is
\begin{equation}
\label{eq:M_{AC}0,1,1,k)}
N(k)=\left\lfloor\frac{1}{6}(2k+3)^2 \right\rfloor,
\end{equation}
and that the bound is attained in the following graphs:
\begin{itemize}
	\item
	For $k=3x$: $T$-shaped tiles with dimensions as the white tile  of Figure
	\ref{fig:general-cases}$(a)$, which we denote as $T_1(x)$; \\
	$\Circ(6x^2+6x+1;\{\pm 1,6x+3\})\cong \Cay(\Z^2/\Z^2M,\{\pm e_1,e_2\})$, with
	$M=$\\
    $\left(\begin{array}{cc}
	3x+2 & -x-1\\
	-1 & 2x+1
	\end{array}\right)$.
	\item For $k=3x-1$: $T$-shaped tiles with dimensions as the shadow tile of Figure \ref{fig:general-cases}$(a)$, which we denote as  $T_2(x)$;\\
	$\Circ(6x^2+2x;\{\pm x,3x+1\})\cong \Cay(\Z^2/\Z^2M,\{\pm e_1,e_2\})$, with
	$M=\left(\begin{array}{cc}
	3x+1 & -x\\
	0 & 2x
	\end{array}\right)$.
	\item
	For $k=3x-2$: $T$-shaped tiles with dimensions as the shadow tile of Figure
	\ref{fig:general-cases}$(c)$, which we denote as  $T_3(x)$;\\
	$\Circ(6x^2-2x;\{\pm x,3x-1\})\cong \Cay(\Z^2/\Z^2M,\{\pm e_1,e_2\})$, with
	$M=\left(\begin{array}{cc}
	3x-1 & -x\\
	-1 & 2x
	\end{array}\right)$.
\end{itemize}

\begin{figure}[t]
	\begin{center}
		\includegraphics[width=14cm]{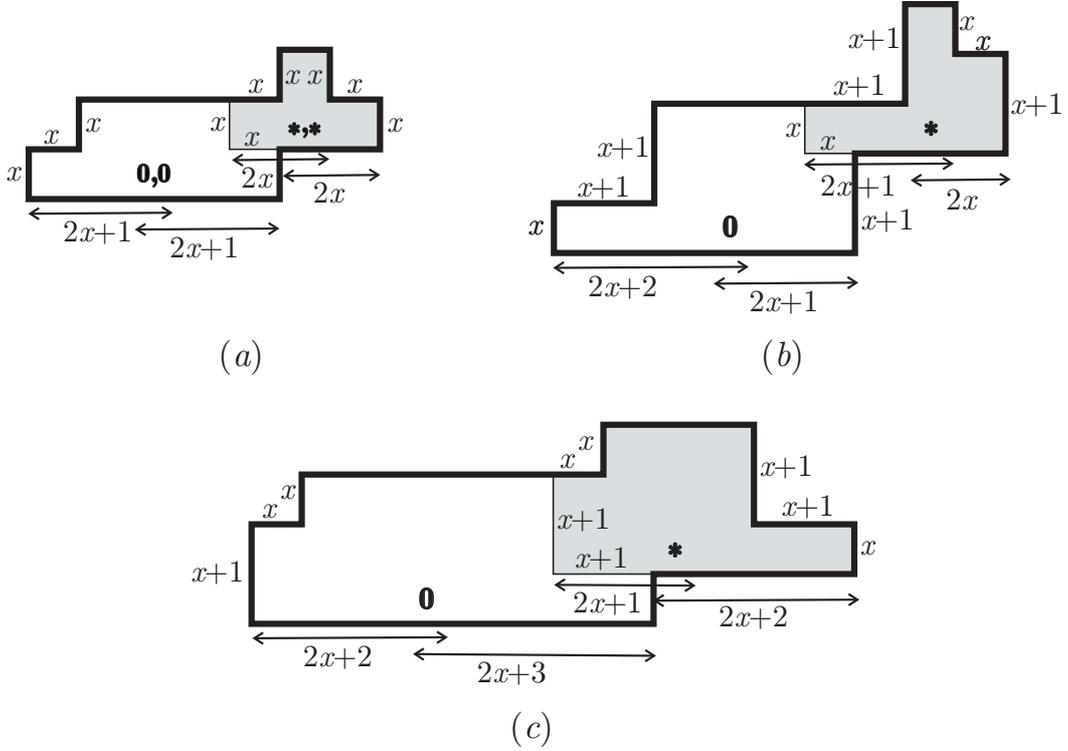}
	\end{center}
	\vskip-.5cm
	\caption{$(a)$ $\Cay(\mathbb{Z}_{6x}\times\mathbb{Z}_{2x},\{(\pm1,0),(0,1),(3x,x)\})$ with $k=3x-1$ ($x\geq 2$) and $N=12x^2$;
		$(b)$ $\Circ(N;\{\pm 1,12x^2+2x-1,6x^2+4x\})$ with $k=3x$ ($x\geq1$) and $N=12x^2+8x$;
		$(c)$ $\Circ(N;\{\pm 1,6x+5,6x^2+8x+2\})$ with $k=3x+1$ ($x\geq1$) and $N=12x^2+16x+4$.}
	\label{fig:general-cases}
\end{figure}

\begin{proposition}
%	Depending on the value of the diameter $k\ge 2$, the mixed Abelian Cayley graphs with $r_{\alpha}=r_{\omega}=z_{\omega}=1$ and maximum order are the following:
Depending on the value of the diameter $k\ge 2$, the maximum order for a mixed Abelian Cayley graph with $r_{\alpha}=r_{\omega}=z_{\omega}=1$ and the graphs attaining it are the following:
\begin{itemize}
\item[$(a_0)$]
         For $k=2$: $N=10$.\\
         $\Circ(10;\{\pm1,2,5\})$.
	    \item[$(a)$]
        For $k=3x-1$ ($x\geq 2$): $N=12x^2$.\\
         $\Cay(\mathbb{Z}_{6x}\times\mathbb{Z}_{2x},\{(\pm1,0),(0,1),(3x,x)\})$.
         		\item[$(b)$]
	For $k=3x$ ($x\geq1$): $N=12x^2+8x$.\\
		$\Circ(N;\{\pm 1,12x^2+2x-1,6x^2+4x\})$.\\
		For example: For $k=3$: $N=20$, $\Circ(N;\{\pm1,13,10\})$.
	\item[$(c)$]
	For $k=3x+1$ ($x\geq1$):	$N=12x^2+16x+4$.\\
		$\Circ(N;\{\pm 1,6x+5,6x^2+8x+2\})$.\\
		For example: For $k=4$: $N=32$, $\Circ(N;\{\pm1,11,16\})$.
	\end{itemize}
\end{proposition}

\begin{figure}[t]
	\begin{center}
		\includegraphics[width=14cm]{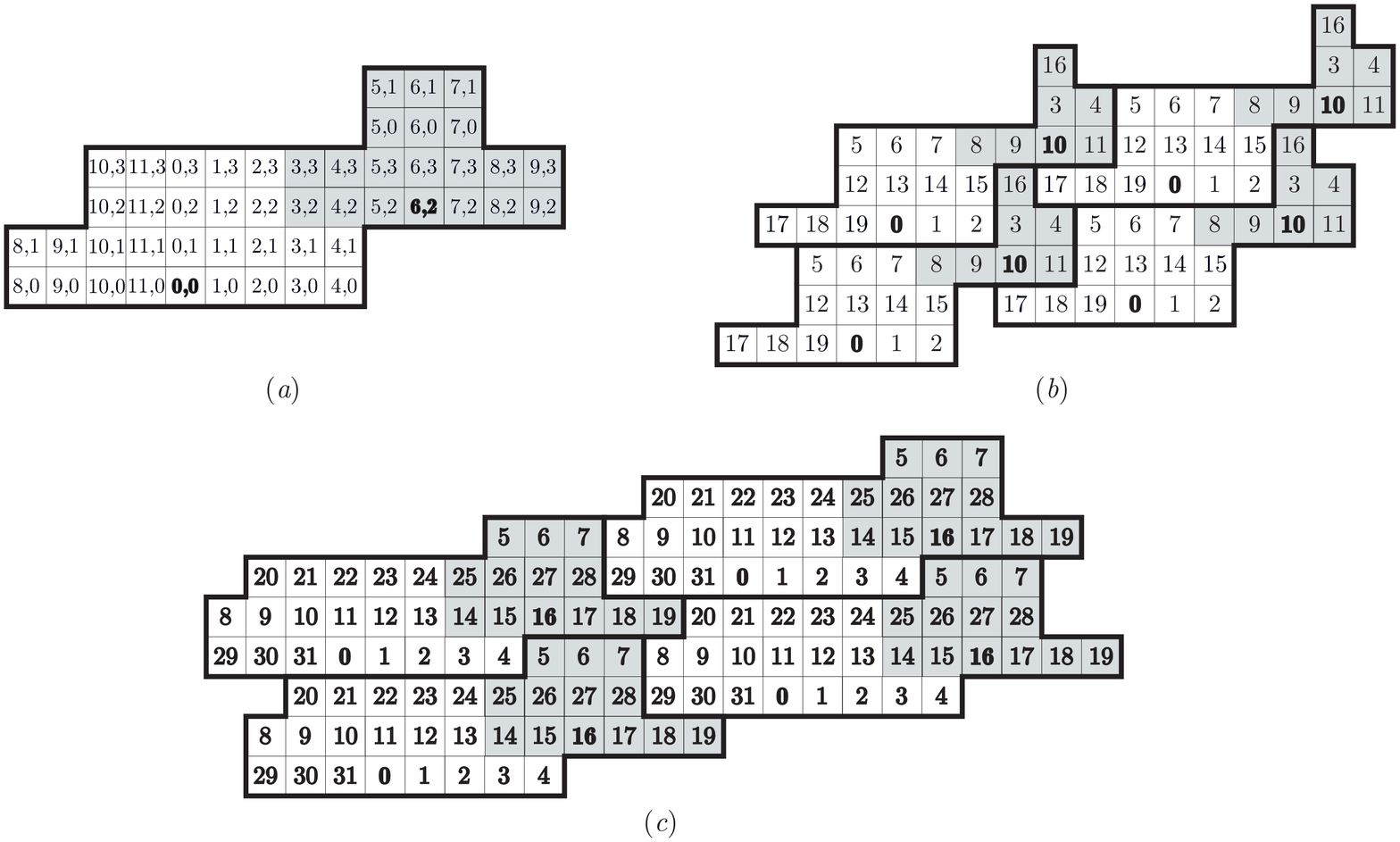}
	\end{center}
	\vskip-.5cm
	\caption{$(a)$ $\Cay(\mathbb{Z}_{12}\times\mathbb{Z}_{4},\{(\pm1,0),(0,1),(6,2)\})$ with $k=5$ and $N=48$;
		$(b)$ $\Circ(20;\{\pm1,13,10\})$ with $k=3$ and $N=20$;
		$(c)$ $\Circ(32;\{\pm1,11,16\})$ with $k=4$ and $N=32$.}
	\label{fig:small-cases}
\end{figure}

\begin{proof}
	First, note that $(a_0)$ is a special case because it not included in $(a)$ since for $x=1$ it turns out that $(0,1)$ and $(3,1)$ are both involutions in $\mathbb{Z}_{6}\times\mathbb{Z}_{2}$. So, $(a)$ would give a graph with $r=4$ and $z=0$.
	For the other (general) cases, we follow the same line of reasoning as in Proposition \ref{propo:diamonds}.  
	Thus, by using the same notation, by a given diameter $k$,	the maximum number of vertices of the vertex sets  $V(0)$ and $V(\ast)$ are given, respectively, by $N(k)$ and $N(k-1)$ in \eqref{eq:M_{AC}0,1,1,k)}. Then, depending on the value of $k$, the maximum number of vertices of a mixed Abelian Cayley graph with $r_{\alpha}=r_{\omega}=z_{\omega}=1$ is attained by joining two $T$-shaped tiles of type $T_i$, $i\in\{1,2,3\}$, as follows:
\begin{itemize}
\item[$(a)$] 
For $k=3x-1$, we take the tiles $T_2(x)$ (with diameter $k$) and $T_3(x)$ (with diameter $k-1$). Notice these tiles correspond, respectively, to the vertices in $V(0)$ and $V(\ast)$. See Figure \ref{fig:general-cases}$(a)$. The `composed' tile $T_2(x)\cup T_3(x)$, with area $N=6x^2$, tessellates the plane with lattice $\mathbb{Z}^2M$, where $M=\left(\begin{array}{cc}
6x & 0\\
0 & 2x
\end{array}\right)$. As the Smith normal form of $M$ is $S(M)=\diag(2x, 6x)$, the group $\Z^2/\Z^2M$ has rank two and it is isomorphic to $\Z_{2x} \times \Z_{6x}$, and the optimal graph is as claimed (with generators $\pm \e_1$, $\e_2$, and the involution $(3x,x)$).
\item[$(b)$] 
For $k=3x$, we should take  the tiles $T_1(x)$ (with diameter $k$) and $T_2(x)$ (with diameter $k-1$). However, this is not possible since the total area is an odd integer. Fortunately, the `composed' tile $T_1(x)\cup T'_2(x)$, where $T'_2(x)$ is a slight modification of $T_2(x)$ with area $N=6x^2+2x-1$ (see Figure \ref{fig:general-cases}$(b)$), 
tessellates the plane with lattice $\Z^2M$, where $M=\left(\begin{array}{cc}
6x+1 & 1\\
1 & 2x+1
\end{array}\right)$. The Smith normal form of this matrix is
\begin{align*}
S &=\diag(1,12x^2+8x)=UMV \\
 &=\left(\begin{array}{cc}
0 & 1\\
-1 & 6x+1
\end{array} \right)\left(\begin{array}{cc}
6x+1 & 1\\
	1 & 2x+1
	\end{array} \right)\left(\begin{array}{cc}
	1 & -2x-1\\
	0 & 1
	\end{array} \right).
\end{align*}
Consequently, the optimal graph on the group $\Z_{12x^2+8x}$ has generators $\{\pm 1, -(2x+1), 6x^2+4x\}$ or, equivalently, $\{\pm 1,6x^2+4x,12x^2+2x-1\}$ (recall that the steps of a circulant graph $G$ can be multiplied with any number relatively prime with $N$ obtaining a graph isomorphic to  $G$), as claimed.
\item[$(c)$] 
For $k=3x+1$, we should take  the tiles $T_3(x+1)$ (with diameter $k$) and $T_1(x)$ (with diameter $k-1$). As in case $(b)$, this is not possible since the total area is an odd integer. Now, the `composed' tile $T_3'(x+1)\cup T_1(x)$, where $T'_3(x+1)$ is a  modification of $T_3(x+1)$ with area $N=12x^2+16x+4$ (see Figure \ref{fig:general-cases}$(c)$)
tessellates the plane with lattice $\Z^2M$, where $M=\left(\begin{array}{cc}
6x+5 & -1\\
-1 & 2x+1
\end{array}\right)$. Then, the Smith normal form is
\begin{align*}
S &=\diag(1,12x^2+16x+4)=UMV \\
&=\left(\begin{array}{cc}
0 & -1\\
1 & 6x+5
\end{array} \right)\left(\begin{array}{cc}
6x+5 & -1\\
-1 & 2x+1
\end{array} \right)\left(\begin{array}{cc}
1 & 2x+1\\
0 & 1
\end{array} \right).
\end{align*}
Consequently, the optimal graph on the group $\Z_{12x^2+10x+4}$ has generators $\{\pm 1, 2x+1, 6x^2+8x+2\}$ or, equivalently, $\{\pm 1,6x+5,6x^2+8x+2\}$, as claimed.
\end{itemize}
	See Figure %\ref{fig:general-cases} for the general cases, and Figure 
	\ref{fig:small-cases} for some examples with diameters $k=3,4,5$.
%	For \textcolor{blue} {NOT} example: For $k=2$: $N=12$, $\Cay(\mathbb{Z}_{6}\times\mathbb{Z}_{2},\{(\pm1,0),(0,1),(3,1)\})$. \textcolor{blue}{(Observeu que aquest exemple és }
\end{proof}

Finally, we point out that the optimal values $N$ may be attained for other mixed graphs in some cases. For instance, in the case ($a_0$) we have that $\Circ(10;\{\pm2,1,5\})$ is also optimal, but it is not isomorphic to $\Circ(10;\{\pm1,2,5\})$.

\section*{Acknowledgments}
\label{sec:acknow}
The first two authors have been partially supported by the project 2017SGR1087 of the Agency for the Management of University and Research Grants (AGAUR) of the Catalan Government, and by MICINN from the Spanish Government under project PGC2018-095471-B-I00. The first and the third authors have been supported in part by grant MTM2017-86767-R of the Spanish Government.

\end{document}